\documentclass[a4paper, 12pt, onecolumn]{ieeeconf}
\usepackage{amssymb}
\usepackage{amsmath}
\usepackage[caption=false,font=footnotesize]{subfig}
\usepackage{amsfonts}
\usepackage{graphicx}

\newtheorem{theorem}{Theorem}
\newtheorem{lemma}{Lemma}
\newtheorem{proposition}{Proposition}
\newtheorem{definition}{Definition}
\newtheorem{algorithm}{Algorithm}
\newtheorem{problem}{Problem}
\newtheorem{example}{Example}

\newcommand{\eps}{\varepsilon}
\newcommand{\supC}{\mbox{$\sup {\rm C}$}}
\newcommand{\infCO}{\mbox{$\inf {\rm C}$}}
\newcommand{\supcn}{\mbox{$\sup {\rm C}$}}

\newcommand{\supccnr}{\mbox{$\sup {\rm rcC}$}}
\newcommand{\supcnr}{\mbox{$\sup {\rm rC}$}}
\newcommand{\supCr}{\mbox{$\sup {\rm rC}$}}

\newenvironment{proofof}[1]{\medskip\noindent \textbf{#1.} }{}

\title{\LARGE \bf
  On a Distributed Computation of Supervisors\\ in Modular Supervisory Control
}

\author{Jan~Komenda, Tom{\' a}{\v s}~Masopust and J.~H.~van~Schuppen%
  \thanks{The authors are with Institute of Mathematics, Academy of Sciences of the Czech Republic, {\v Z}i{\v z}kova 22, 616 62 Brno, Czech Republic, Technische Universit\"at Dresden, Germany, and Van Schuppen Control Research, Gouden Leeuw 143, 1103 KB Amsterdam, The Netherlands, resp.
  {\tt\small komenda{@}ipm.cz, tomas.masopust{@}tu-dresden.de, jan.h.van.schuppen@xs4all.nl}
  }%
}

\begin{document}

\maketitle
\thispagestyle{empty}
\pagestyle{empty}

\begin{abstract}
  In this paper, we discuss a supervisory control problem of modular discrete-event systems that allows for a distributed computation of supervisors. We provide a characterization and an algorithm to compute the supervisors. If the specification does not satisfy the properties, we make use of a relaxation of coordination control to compute a sublanguage of the specification for which the supervisors can be computed in a distributed way.
\end{abstract}

\section{Introduction and motivation}
  We investigate distributed supervisory control of concurrent discrete-event systems. Supervisory control theory of discrete-event systems modeled as finite automata was introduced by Ramadge and Wonham~\cite{RW87} and studied by many others. It aims to guarantee that the control specifications consisting of safety and of nonblockingness are satisfied in the controlled (closed-loop) system. Safety means that the language of the closed-loop system is included in a prescribed specification, and nonblockingness means that all controlled behaviors can always be completed to a marked controlled behavior. Supervisory control is realized by a supervisor that runs in parallel with the system and imposes the specification by disabling, at each state, some of the controllable events in a feedback manner. Since only controllable specifications can be achieved, one of the key issues is the computation of the supremal controllable sublanguage of the specification, from which the supervisor can be constructed.
  
  Supervisory control theory is well developed for monolithic systems, i.e., systems where the plant is modeled as a single generator. However, most of the current complex engineering systems can be abstracted as a composition of many components. Systems that model technological systems typically consist of small generators that communicate with each other in a synchronous~\cite{CL08} or asynchronous~\cite{LauriePhilippe} way. Such systems are often called modular (or concurrent or distributed) discrete-event systems. It is known that to compute the overall monolithic plant for such a system can be unrealistic because the number of states of a modular system grows exponentially with respect to the number of components. This limits the applicability of the monolithic supervisory control synthesis to relatively small systems. On the other hand, the purely decentralized control consisting of an independent construction of a supervisor for each subsystem is only applicable for a local (decomposable) specification. 
  
  Specifically, let $G_1$ and $G_2$ be two systems over the respective alphabets $\Sigma_1$ and $\Sigma_2$ modeled as finite generators forming the overall plant $G_1\| G_2$, the computation of which we want to avoid, and let $K\subseteq L_m(G_1\| G_2)$ denote a specification. There exists a monolithic supervisor for the monolithic plant $G_1\| G_2$ if and only if the specification is controllable with respect to the plant. However, to avoid the computation of the monolithic plant $G_1\| G_2$, the naive approach to compute supervisors for each subsystem separately does not work in general. To demonstrate this, assume that we compute supervisors $S_i$ such that $L_m(S_i/G_i) = P_i(K)$, where $P_i(K)$ denotes the projection of the specification $K$ to the alphabet of the generator $G_i$. Then, it can still be the case that $K \subsetneq P_1(K) \parallel P_2(K)$. The problem here is that the specification is not necessarily decomposable. Recall that a languages $L$ is decomposable (or separable) with respect to alphabets $\Sigma_1$ and $\Sigma_2$ if $L = P_1(L) \parallel P_2(L)$, where $P_i:(\Sigma_1\cup\Sigma_2)^* \to \Sigma_i^*$ is a natural projection (see more details below). If a language is not decomposable, it would be natural to search for some ``maximal'' decomposable sublanguage. Unfortunately, this is not algorithmically possible~\cite{LinSSWS14}. 

  To deal with this problem, we introduced the notion of {\em conditional decomposability}. Conditional decomposability of a language $L$ with respect to alphabets $\Sigma_1$ and $\Sigma_2$ requires to find another alphabet $\Sigma_k$ such that the language $L$ is decomposable with respect to $\Sigma_1\cup\Sigma_k$ and $\Sigma_2\cup\Sigma_k$, i.e., $L=P_{1+k}(L) \parallel P_{2+k}(L)$, see more details below. The difference between decomposability and conditional decomposability is that every language can be made conditionally decomposable, which overcomes the undecidable problem of finding a nonempty decomposable sublanguage discussed above. In other words, we can always find an alphabet $\Sigma_k$ such that the language becomes conditionally decomposable. Indeed, one could take $\Sigma_k = \Sigma_1 \cup \Sigma_2$, but the aim is to find a reasonably small such alphabet. Although to find a minimal alphabet is NP-hard~\cite{JDEDS}, a polynomial-time algorithm to find an acceptable alphabet is described in~\cite{scl12}. 
  
  The observation that conditional decomposability is always possible leads us to the formulation of the following problem that we investigate in this paper. 
  \begin{problem}\label{problem1}
    Let $G_1$ and $G_2$ be generators over the alphabets $\Sigma_1$ and $\Sigma_2$, respectively. Assume that a specification $K \subseteq L_m(G_1 \| G_2)$ and its prefix-closure $\overline{K}$ are conditionally decomposable with respect to $\Sigma_1$, $\Sigma_2$, and $\Sigma_k$, for some alphabet $\Sigma_k$ such that $\Sigma_1 \cap \Sigma_2 \subseteq \Sigma_k \subseteq \Sigma_1\cup\Sigma_2$. Let $G_k = P_k(G_1) \parallel P_k(G_2)$ be a coordinator that ensures the necessary communication between the systems $G_1$ and $G_2$. The aim is to determine nonblocking supervisors $S_{1}$ and $S_{2}$ such that the supervised distributed system satisfies
    \begin{align*}
      L_m(S_{1}/ [G_1 \| G_k ]) \parallel L_m(S_{2}/ [G_2 \| G_k ]) = K
      && \text{ and } &&
      L(S_{1}/ [G_1 \| G_k ]) \parallel L(S_{2}/ [G_2 \| G_k ]) = \overline{K}\,,
    \end{align*}
    i.e., it fulfills the specification and the supervisors are nonconflicting. \QEDopen
  \end{problem}
  
  In other words, the question is whether it is possible to distribute the monolithic (global) supervisor with help of a coordinator in such a way that there exist nonblocking and nonconflicting supervisors $S_1$ and $S_2$ such that the supervised distributed closed-loop system satisfies safety and nonblockingness.
  
  Note that we restrict ourselves to the case, where the coordinator $G_k$ does not affect the behavior of the whole system. The case where $G_k$ can influence the system, i.e., the new plant is $G_1\| G_2\| G_k$, is left for the future investigation. However, it should be pointed out that one has to be careful with this modification, since the coordinator could influence the system in such a way that it forbids some uncontrollable events, which may be physically impossible. A possible solution would be to consider only such coordinators $G_k$ that are controllable with respect to (some projection of) $G_1\| G_2$. 

  In this paper, we investigate Problem~\ref{problem1}. We provide a necessary and sufficient condition for the existence of a solution, which results in a decision procedure and in an algorithm to compute the solution, if it exists. As usual, if the solution does not exist, we make use of the coordination control methods to construct an acceptable sub-specification for which a solution exists, measured with respect to the solution provided in the coordination control framework. Moreover, we relax the coordination control framework so that the main results of the framework still hold, but the presentation is simplified and potentially larger solutions can be achieved. However, we do not claim that the use of the coordination control framework is the best here, we only show that it is possible to use it here. It is an interesting research topic whether other methods could be used to obtain a better solution. Related work that could be useful to achieve this includes, but is not limited to, \cite{FLT,hubbard:caines:2002,KS,WZ91}.
  
  The paper is organized as follows. In Section~\ref{sec:preliminaries}, the necessary notation is introduced and the basic elements of supervisory control theory are recalled. In Section~\ref{sec:existenceOFsolution}, we provide necessary and sufficient conditions under which the solution of the problem exists. In Section~\ref{sec:relaxation}, we recall the basic coordination control framework and introduce its relaxed variant. In Section~\ref{sec:comparison}, we compare Problem~\ref{problem1} with the (relaxed) coordination control problem. In Section~\ref{sec:app}, we make use of the relaxed coordination control to find a reasonable sub-specification for which a solution exists if a solution for the specification does not exist. An industrial example is provided in Section~\ref{sec:MRI}.
  
  We discuss only the full-observation case here. The approach can be extended to partial observations using either the notion of (conditional) normality of~\cite{relaxed} or the notion of (conditional) relative observability of~\cite{caiCDC13,KomendaMS14a}.

\section{Preliminaries and definitions}\label{sec:preliminaries}
  Let $\Sigma$ denote a finite nonempty set of events (an {\em alphabet}), and let $\Sigma^*$ denote the set of all finite words over $\Sigma$. The empty word is denoted by $\eps$. A {\em language\/} over $\Sigma$ is a subset of $\Sigma^*$. The prefix closure of a language $L$ over $\Sigma$ is the set $\overline{L}=\{w\in \Sigma^* \mid \text{there exists } u \in\Sigma^* \text{ such that } wu\in L\}$. A language $L$ is {\em prefix-closed\/} if $L=\overline{L}$. For more details, the reader is referred to~\cite{CL08,Won12}. 
  
  A {\em generator\/} is a structure $G=(Q,\Sigma, f, q_0, Q_m)$, where $Q$ is the finite set of states, $\Sigma$ is the finite nonempty set of events, $f: Q \times \Sigma \to Q$ is the partial transition function, $q_0 \in Q$ is the initial state, and $Q_m\subseteq Q$ is the set of marked states. The transition function $f$ can be extended to the domain $Q \times \Sigma^*$ in the usual way. The behavior of $G$ is described in terms of languages. The language {\em generated\/} by $G$ is the set $L(G) = \{s\in \Sigma^* \mid f(q_0,s)\in Q\}$, and the language {\em marked\/} by $G$ is the set $L_m(G) = \{s\in \Sigma^* \mid f(q_0,s)\in Q_m\}$. Obviously, $L_m(G)\subseteq L(G)$.

  For $\Sigma_0\subseteq \Sigma$, a {\em (natural) projection\/} is a mapping $P: \Sigma^* \to \Sigma_0^*$, which deletes from any word all letters that belong to $\Sigma\setminus \Sigma_0$. Formally, it is a homomorphism defined by $P(a)=\eps$, for $a$ in $\Sigma\setminus \Sigma_0$, and $P(a)=a$, for $a$ in $\Sigma_0$. It is extended (as a homomorphism for concatenation) from letters to words by induction. The {\em inverse image\/} of $P$ is denoted by $P^{-1}:2^{\Sigma_0^*} \to 2^{\Sigma^*}$. For three alphabets $\Sigma_i$, $\Sigma_j$, $\Sigma_\ell$, subsets of $\Sigma$, the notation $P^{i+j}_{\ell}$ denotes the projection from $(\Sigma_i\cup \Sigma_j)^*$ to $\Sigma_\ell^*$. If $\Sigma_i\cup \Sigma_j=\Sigma$, we simplify the notation to $P_\ell$. Similarly, the notation $P_{i+k}$ stands for the projection from $\Sigma^*$ to $(\Sigma_i\cup\Sigma_k)^*$. The projection of a generator $G$, denoted by $P(G)$, is a generator whose behavior satisfies $L(P(G))=P(L(G))$ and $L_m(P(G))=P(L_m(G))$. It is defined using the standard subset construction, cf.~\cite{CL08}.
 
  The synchronous product of languages $L_1$ over $\Sigma_1$ and $L_2$ over $\Sigma_2$ is the language $L_1\parallel L_2=P_1^{-1}(L_1) \cap P_2^{-1}(L_2)$, where $P_i: (\Sigma_1\cup \Sigma_2)^*\to \Sigma_i^*$ is a projection, $i=1,2$. A definition for generators can be found in~\cite{CL08}. For generators $G_1$ and $G_2$, $L(G_1 \| G_2) = L(G_1) \parallel L(G_2)$ and $L_m(G_1 \| G_2)= L_m(G_1) \parallel L_m(G_2)$. 
  Languages $K$ and $L$ are {\em synchronously nonconflicting\/} if $\overline{K \parallel L} = \overline{K} \parallel \overline{L}$. 
  
  We recall the basic elements of supervisory control theory. A {\em controlled generator\/} over an alphabet $\Sigma$ is a structure $(G,\Sigma_c,\Gamma)$, where $G$ is a generator over the alphabet $\Sigma$, $\Sigma_c\subseteq\Sigma$ is the set of {\em controllable events}, $\Sigma_{u} = \Sigma \setminus \Sigma_c$ is the set of {\em uncontrollable events}, and $\Gamma = \{\gamma \subseteq \Sigma \mid \Sigma_{u} \subseteq \gamma\}$ is the {\em set of control patterns}. A {\em supervisor\/} for the controlled generator $(G,\Sigma_c,\Gamma)$ is a map $S:L(G) \to \Gamma$. The {\em closed-loop system\/} associated with the controlled generator $(G,\Sigma_c,\Gamma)$ and the supervisor $S$ is defined as the smallest language $L(S/G)$ such that $\eps \in L(S/G)$, and if $s \in L(S/G)$, $sa\in L(G)$, and $a \in S(s)$, then $sa \in L(S/G)$. The marked language of the closed-loop system is defined as $L_m(S/G) = L(S/G)\cap K$, where $K\subseteq L_m(G)$ is a specification. The intuition is that the supervisor disables some of the controllable transitions of $G$, but never an uncontrollable transition, and marks in accordance with the specification. If $\overline{L_m(S/G)}=L(S/G)$, then the supervisor $S$ is called {\em nonblocking}. 
  
  In the automata framework, where a supervisor $S$ has a finite representation as a generator, basically it is the generator $G$ with some of the controllable events disabled, the closed-loop system is a synchronous product of the supervisor and the plant. Thus, we can write the closed-loop system as $L(S/G)=L(S) \parallel L(G) = L(S)$. Moreover, the supervisor keeps the information about the marked states, i.e., we have that $L_m(S/G)=L_m(S) \parallel L_m(G) = L_m(S)$. The supervisor is then nonblocking if $\overline{L_m(S)} = L(S)$.

  Control objectives of supervisory control are defined using a specification $K$. The goal of supervisory control is to find a nonblocking supervisor $S$ such that $L_m(S/G)=K$. In the monolithic case, such a supervisor exists if and only if $K$ is {\em controllable\/} with respect to $L(G)$ and $\Sigma_u$ (i.e., $\overline{K}\Sigma_u\cap L\subseteq \overline{K}$), cf.~\cite{CL08,Won12}.

  If the specification fails to satisfy controllability, controllable sublanguages of the specification are considered instead. Let $\supC(K,L(G),\Sigma_u)$ denote the supremal controllable sublanguage of $K$ with respect to $L(G)$ and $\Sigma_u$, which always exists and is equal to the union of all controllable sublanguages of $K$, cf.~\cite{Won12}. 

  The projection $P:\Sigma^* \to \Sigma_0^*$, with $\Sigma_0\subseteq \Sigma$, is an {\em $L$-observer\/} for $L\subseteq \Sigma^*$ if, for all $t\in P(L)$ and $s\in \overline{L}$, $P(s)$ is a prefix of $t$ implies that there exists $u\in \Sigma^*$ such that $su\in L$ and $P(su)=t$, see~\cite{wong98,pcl08}. This property is well known and widely used in supervisory control of hierarchical and distributed discrete-event systems, and, as mentioned in~\cite{pena2010}, also in compositional verification~\cite{FM2009} and modular synthesis~\cite{FW2006,HT2006}. If the projection does not satisfy the property, the co-domain of the projection can be extended so that it is satisfied. Although the computation of such a minimal extension is NP-hard, there exists a polynomial-time algorithm that finds an acceptable extension~\cite{FengWonham}. More information can be found in~\cite{pena2010}. 
 
  A language $K$ is {\em conditionally decomposable\/} with respect to alphabets $\Sigma_1$, $\Sigma_2$, and $\Sigma_k$, where $\Sigma_1\cap\Sigma_2\subseteq\Sigma_k\subseteq \Sigma_1\cup\Sigma_2$, if 
  \[
    K = P_{1+k}(K)\parallel P_{2+k}(K)
  \]
  where $P_{i+k}:(\Sigma_1\cup\Sigma_2)^*\to (\Sigma_i\cup\Sigma_k)^*$ is a projection, for $i=1,2$. 
  The definition is extended to more alphabets in a natural way.

  For simplicity, we formulated Problem~\ref{problem1} only for two subsystems. However, it is clear how to extend it to $n>2$ subsystems. It should perhaps be pointed out that for $n>2$, the set $\Sigma_k$ contains all shared events $\bigcup_{i\neq j} (\Sigma_i\cap\Sigma_j)$.
  
  We define the coordinator $G_k$ as the parallel composition of projections. It is a folklore in automata theory that, in the worst case, the computation of a projected generator can be of exponential size with respect to the number of states of the original generator. However, it is known that if the projection satisfies the observer property, then the projected generator is smaller than the original generator~\cite{wong98}. Therefore, we suggest the following steps to compute the alphabet $\Sigma_k$ and the coordinator $G_k$:
  \begin{enumerate}
    \item Let $\Sigma_k = \Sigma_1\cap \Sigma_2$ be the set of all shared events of the generators $G_1$ and $G_2$.
    \item Extend the alphabet $\Sigma_k$ so that $K$ and $\overline{K}$ are conditionally decomposable with respect to $\Sigma_1$, $\Sigma_2$, and $\Sigma_k$, cf.~\cite{scl12} for a polynomial algorithm.
    \item If the computation of $G_k$ in the following step is not possible (e.g., because of the state explosion), extend the alphabet $\Sigma_k$ so that the projection $P_k:\Sigma^* \to \Sigma_k^*$ is an $L(G_i)$-observer, for $i=1,2$, cf.~\cite{pcl08,pcl12}.
    \item Define the coordinator $G_k$ as $G_k = P_k(G_1) \parallel P_k(G_2)$, where $P_k:(\Sigma_1\cup\Sigma_2)^*\to \Sigma_k^*$.
  \end{enumerate}

\section{Existence of a solution}\label{sec:existenceOFsolution}
  We now provide necessary and sufficient conditions for the existence of a solution for Problem~\ref{problem1}.

  \begin{lemma}[Characterization]\label{thm:char}
    Consider the setting of Problem~\ref{problem1}. There are nonblocking and nonconflicting supervisors $S_{1}$ and $S_{2}$ such that the distributed closed-loop system satisfies 
    $
      L_m(S_{1}/[G_1 \| G_k]) \parallel L_m(S_{2}/[G_2 \| G_k]) = K
    $ 
    if and only if
      $L_m(S_1/(G_1\| G_k)) \parallel P_k(S_2) = P_{1+k}(K)$,
      $L_m(S_2/(G_2\| G_k)) \parallel P_k(S_1) = P_{2+k}(K)$, and
      $S_1$ and $S_2$ are nonblocking and nonconflicting supervisors with respect to $G_1\| G_k$ and $G_2\| G_k$, respectively. \QED
  \end{lemma}
  
  The second part of the theorem gives us the following equations with variable languages $X_1$ and $X_2$:
  \begin{equation}\label{sol1}
    \begin{aligned}
      P_{1+k}(K) \subseteq X_1 \subseteq G_1\| G_k, && X_1 \cap (P^{1+k}_k)^{-1}P_k(X_2) &= P_{1+k}(K) \\
      P_{2+k}(K) \subseteq X_2 \subseteq G_2\| G_k, && X_2 \cap (P^{2+k}_k)^{-1}P_k(X_1) &= P_{2+k}(K)\,.
    \end{aligned}
  \end{equation}
  
  Let $\infCO(K,L)$ denote the infimal {\em prefix-closed\/} controllable superlanguage of $K$ with respect to $L$, cf.~\cite{CL08}. Then we immediately have the following.
  
  \begin{proposition}[Prefix-closed specification]\label{lem1}
    Consider the setting of Problem~\ref{problem1}. If the specification $K$ is prefix-closed, then there exists a solution of Problem~\ref{problem1} if and only if the languages 
      $T_1 = \infCO(P_{1+k}(K), G_1\| G_k)$ and
      $T_2 = \infCO(P_{2+k}(K), G_2\| G_k)$
    satisfy equations~(\ref{sol1}). \QED
  \end{proposition}

  The infimal controllable superlanguage is prefix-closed and it is not possible to compute the infimal non-prefix-closed controllable superlanguage, since it does not always exist. Thus, the previous lemma cannot be directly used for non-prefix-closed languages. However, we show that it is surprisingly sufficient to verify whether there exists a solution for the prefix-closure of the specification. 
  \begin{theorem}\label{thmMain}
    Consider the setting of Problem~\ref{problem1}. There exists a solution for specification $K$ if and only if there exists a solution for its prefix-closure $\overline{K}$. \QED
  \end{theorem}
  
  In other words, assuming that $K$ and $\overline{K}$ are conditionally decomposable, it is sufficient to check whether there exists a solution for the prefix-closure $\overline{K}$ of the specification $K$ using Proposition~\ref{lem1}. By Theorem~\ref{thmMain}, this implies a solution for the specification $K$, which can actually be derived from the solution for $\overline{K}$.

\section{Relaxation of coordination control}\label{sec:relaxation}
  If a solution of Problem~\ref{problem1} does not exist, we make use of the coordination control framework to obtain an acceptable sublanguage of the specification for which a solution exists. 
  
  Coordination control is based on the concept of conditional independence of probability theory, which allows for coordinated systems not only in discrete-event systems, but also in linear, stochastic, and hybrid systems, see~\cite{schuppen:etal:2011:ejc}. 

  We first recall the original coordination control framework and then relax it so that it admits more and larger solutions, cf.~Example~\ref{ex1}. 
  
\subsection{Original coordination control framework}
  We briefly recall the original coordination control framework developed in~\cite{JDEDS}. The notation is that of~\cite{JDEDS}.
  
  \begin{problem}[Coordination control problem]\label{problem:Origcontrolsynthesis}
    Consider generators $G_1$ and $G_2$ over the alphabets $\Sigma_1$ and $\Sigma_2$, resp., and a generator $G_k$ (a {\em coordinator\/}) over an alphabet $\Sigma_k$, where $\Sigma_1\cap\Sigma_2\subseteq \Sigma_k$. Assume that a specification $K \subseteq L_m(G_1 \| G_2 \| G_k)$ and its prefix-closure $\overline{K}$ are conditionally decomposable with respect to $\Sigma_1$, $\Sigma_2$, and $\Sigma_k$. The aim is to determine nonblocking supervisors $S_1$, $S_2$, and $S_k$ such that $L_m(S_k/G_k)\subseteq P_k(K)$, $L_m(S_i/ [G_i \| (S_k/G_k) ])\subseteq P_{i+k}(K)$, for $i=1,2$, and the closed-loop system with the coordinator satisfies $L_m(S_1/ [G_1 \| (S_k/G_k) ]) \parallel L_m(S_2/ [G_2 \| (S_k/G_k) ]) = K$.
    \QEDopen
  \end{problem}

  Although the coordinator $G_k$ in the statement of Problem~\ref{problem:Origcontrolsynthesis} is a general automaton, we suggested to construct it in the same way as in Problem~\ref{problem1}. Therefore, we may assume that $G_k = P_k(G_1) \parallel P_k(G_2)$. In what follows, we use the notation $\Sigma_{i,u} = \Sigma_i \cap \Sigma_u$ to denote the set of locally uncontrollable events of the alphabet $\Sigma_i$.

  \begin{definition}[Conditional controllability]\label{def:Origconditionalcontrollability}
    Let $G_1$ and $G_2$ be generators over $\Sigma_1$ and $\Sigma_2$, respectively, and let $G_k$ be a coordinator over $\Sigma_k$. A language $K\subseteq L_m(G_1\| G_2\| G_k)$ is {\em conditionally controllable\/} for generators $G_1$, $G_2$, $G_k$ and uncontrollable alphabets $\Sigma_{1,u}$, $\Sigma_{2,u}$, $\Sigma_{k,u}$ if
      (i) $P_k(K)$ is controllable with respect to $L(G_k)$ and $\Sigma_{k,u}$ and
      (ii) $P_{i+k}(K)$ is controllable with respect to $L(G_i) \parallel \overline{P_k(K)}$ and $\Sigma_{i+k,u}$,
    where $\Sigma_{i+k,u} = (\Sigma_i\cup \Sigma_k)\cap \Sigma_u$, for $i=1,2$. \QEDopen
  \end{definition}
  
  The main existential result follows.
  \begin{theorem}[\cite{JDEDS}]
    Consider the setting of Problem~\ref{problem:Origcontrolsynthesis}. There exist nonblocking supervisors $S_1$, $S_2$, and $S_k$ such that $L_m(S_1/[G_1 \| (S_k/G_k)]) \parallel L_m(S_2/[G_2 \| (S_k/G_k)]) =  K$ if and only if the specification $K$ is conditionally controllable with respect to generators $G_1$, $G_2$, $G_k$ and uncontrollable alphabets $\Sigma_{1,u}$, $\Sigma_{2,u}$, $\Sigma_{k,u}$. \QED
  \end{theorem}
  
  Similarly as in the monolithic supervisory control, if the specification fails to be conditionally controllable, the supremal conditionally controllable sublanguage is computed. It always exists~\cite{JDEDS}. Consider the setting of Problem~\ref{problem:Origcontrolsynthesis} and define the languages
  \begin{align*}
      \supC_k     & =  \supC(P_k(K), L(G_k), \Sigma_{k,u})\\
      \supC_{i+k} & =  \supC(P_{i+k}(K), L(G_i) \parallel \overline{\supC_k}, \Sigma_{i+k,u})
  \end{align*}
  for $i=1,2$. If $\supC_{1+k}$ and $\supC_{2+k}$ are nonconflicting and $P_k(\supC_{1+k}) \cap P_k(\supC_{2+k})$ is controllable with respect to $L(G_k)$ and $\Sigma_{k,u}$, then $\supC_{1+k} \parallel \supC_{2+k}$ is the supremal conditionally controllable sublanguage of $K$. Otherwise, we compute a new supervisor
  \begin{align*}
      \supC(P_k(\supC_{1+k})\cap P_k(\supC_{2+k}), L(G_k), \Sigma_{k,u})
  \end{align*}
  denoted $\supC_k'$. This gives the following result, which is not stated in~\cite{JDEDS}, but is in a more general form stated in~\cite{case15}.
  \begin{theorem}\label{thm4}
    Consider the setting of Problem~\ref{problem:Origcontrolsynthesis} and the languages defined above. If $\supC_{i+k}$ and $\supC_{k}'$ are synchronously nonconflicting (e.g., prefix-closed) for $i=1,2$, then $\supC'_k \parallel \supC_{1+k} \parallel  \supC_{2+k}$ is the supremal conditionally controllable sublanguage of $K$. \QED
  \end{theorem}

\subsection{Relaxed coordination control framework}  
  In this section, we relax the coordination control problem. Proofs of the results presented in this section can be found in the technical report~\cite{relaxed}. Unlike the original approach, the supervisor for the coordinator is not computed and, thus, it is not included in the closed-loop system. We comment on this below the definition.
 
  \begin{problem}[Relaxed coordination control problem]\label{problem:relaxed}
    Consider generators $G_1$ and $G_2$ over the alphabets $\Sigma_1$ and $\Sigma_2$, respectively. Let $\Sigma_k$ be an alphabet such that $\Sigma_1 \cap \Sigma_2 \subseteq \Sigma_k \subseteq \Sigma_1\cup\Sigma_2$. A generator $G_k$ over the alphabet $\Sigma_k$ is called a coordinator. Assume that a specification $K \subseteq L_m(G_1 \| G_2 \| G_k)$ and its prefix-closure $\overline{K}$ are conditionally decomposable with respect to $\Sigma_1$, $\Sigma_2$, and $\Sigma_k$. The aim  is to determine nonblocking supervisors $S_1$ and $S_2$ such that $L_m(S_i/ [G_i \| G_k ])\subseteq P_{i+k}(K)$, for $i=1,2$, and the closed-loop system satisfies $L_m(S_1/ [G_1 \| G_k ]) \parallel L_m(S_2/ [G_2 \| G_k ]) = K$. \QEDopen
  \end{problem}

  The construction of a coordinator remains unchanged, compared to the original framework.
  
  The difference between the original and the relaxed problem is that the supervisor $S_k$ and, hence, the closed-loop system $S_k/G_k$ for the coordinator part of the specification is not computed. The original motivation to include the supervisor $S_k$ was the anti-monotonicity of the basic supervisory control operator (the supremal controllable sublanguage). Indeed, if the specification is fixed, then making the plant language smaller results in increasing permissiveness of the supervisors given by the supremal controllable sublanguages. However, it turns out that decreasing the plant $G_i\| G_k$ for the local supervisors by replacing $G_k$ with, in general, a smaller closed-loop system $S_k/G_k$ does not change the permissiveness of the local closed-loop systems $S_{i}/[G_i \| (S_k/G_k)]$. This is because of the transitivity of controllability, since the closed-loop $S_k/G_k$ is always controllable with respect to $G_k$. This means that a language $M \subseteq G_i \| (S_k/G_k)$ is controllable with respect to $G_i \| (S_k/G_k)$ if and only if it is controllable with respect to $G_i \| G_k$. For this reason, the supervisor for the coordinator does not help and is not used in the relaxed framework. 
  
  \begin{definition}[Relaxed conditional controllability]\label{def:conditionalcontrollability}
    Let $G_1$ and $G_2$ be generators over the alphabets $\Sigma_1$ and $\Sigma_2$, respectively, and let $G_k$ be a coordinator over the alphabet $\Sigma_k$. A language $K\subseteq L_m(G_1\| G_2\| G_k)$ is {\em relaxed conditionally controllable\/} for generators $G_1$, $G_2$, $G_k$ if the projected specification $P_{i+k}(K)$ is controllable with respect to $L(G_i \| G_k)$ and $\Sigma_{i+k,u}$, where $\Sigma_{i+k,u}=(\Sigma_i\cup \Sigma_k)\cap \Sigma_u$, for $i=1,2$. \QEDopen
  \end{definition}

  By definition, any solution of the coordination control problem is also a solution of the relaxed coordination control problem. The opposite does not hold as demonstrated in the following simple example. 
  \begin{example}\label{ex1}
    Let $L_m(G_1)=\{aa_1bd\}$ and $L_m(G_2)=\{aa_2cd\}$, where the set of controllable events is $\Sigma_c=\{b,c\}$. Let the specification $K$ be the composition of $K_1 = \{aa_1\}$ and $K_2=\{aa_2\}$. Then $K$ it is conditionally decomposable with respect to $\Sigma_1=\{a,a_1,b,d\}$, $\Sigma_2=\{a,a_2,c,d\}$, and $\Sigma_k=\{a,d\}$. 
    Notice that $P_k(K)=\{a\}$ and $G_k=P_k(G_1)\parallel P_k(G_2) = \{ad\}$. Therefore, $P_k(K)$ is {\em not} controllable with respect to $L(G_k)$, thus it is {\em not} conditionally controllable and, hence, not a solution of the original coordination control problem. Moreover, there does not exist a nonempty controllable sublanguage of $P_k(K)$ that is controllable with respect to $L(G_k)$. 
    On the other hand, $P_{i+k}(K)=K_i$ and $L_m(G_i\| G_k)=L_m(G_i)$, hence $P_{i+k}(K)$ is controllable with respect to $L(G_i\| G_k)$, for $i=1,2$. Since nothing is required for $P_k(K)$ in the relaxed framework, $K$ is a solution of the relaxed coordination control problem.
    \QEDopen
  \end{example}
  
  The advantage of the relaxed framework is thus not only a simplification of the notation, but also the fact that there exist solutions in the relaxed framework that cannot be achieved in the original coordination control framework. However, we should mention that the relaxed framework also restricts the set of potential solutions, i.e., there exist solutions of Problem~\ref{problem1} that cannot be achieved in the relaxed framework. This will be clarified in Section~\ref{sec:comparison}. Therefore, further relaxations of this framework would be of interest. 
  
  We now show that the main results of the original framework hold in the relaxed framework, too. 
 
  \begin{theorem}\label{th:controlsynthesissafety}
    Consider the setting of Problem~\ref{problem:relaxed}. There exist nonblocking supervisors $S_1$ and $S_2$ such that the closed-loop system satisfies $L_m(S_1/[G_1 \| G_k]) \parallel L_m(S_2/[G_2 \| G_k]) = K$ if and only if the specification $K$ is relaxed conditionally controllable for generators $G_1$, $G_2$, $G_k$. \QED
  \end{theorem}
  
  If the specification cannot be achieved as the resulting behavior of the coordinated system according to Theorem~\ref{th:controlsynthesissafety}, we describe a procedure to compute the maximal sublanguage of the specification that satisfies both these conditions. We use the notation
  \begin{equation}\label{supccn}
    \supccnr = \supccnr(K, L, (\Sigma_{1,u}, \Sigma_{2,u}, \Sigma_{k,u}))
  \end{equation}
  to denote the supremal relaxed conditionally controllable sublanguage of $K$ with respect to plant $L=L(G_1\| G_2\| G_k)$ and the sets of uncontrollable events $\Sigma_{1,u}$, $\Sigma_{2,u}$, $\Sigma_{k,u}$, which exists and equals to the union of all relaxed conditionally controllable sublanguages of the language $K$. Furthermore, we define the supremal controllable languages for the local plants combined with the coordinator. Consider the setting of Problem~\ref{problem:relaxed} and define the languages
  \begin{equation}\label{eqCNr}
    \begin{aligned}
      \supcnr_{1+k} & = \supcn(P_{1+k}(K), L(G_1\| G_k), \Sigma_{1+k,u})\\
      \supcnr_{2+k} & = \supcn(P_{2+k}(K), L(G_2\| G_k), \Sigma_{2+k,u})
    \end{aligned}
  \end{equation}
  where $\supcn(K,L,\Sigma_u)$ denotes the supremal controllable sublanguage of $K$ with respect to $L$ and $\Sigma_u$. 

  We now generalize the sufficient conditions for a distributed computation of the supremal conditionally controllable sublanguage from the original to the relaxed framework. The sufficient conditions are formulated in terms of controllability of the composition of local supervisors projected to the coordinator alphabet. The first main constructive result of the relaxed coordination control framework to guarantee relaxed conditional controllability in a maximally permissive way is stated below.
  
  \begin{theorem}\label{thm2b}
    Consider the setting of Problem~\ref{problem:relaxed} and the languages defined in~(\ref{eqCNr}). If $\supcnr_{1+k}$ and $\supcnr_{2+k}$ are synchronously nonconflicting and $P_k(\supcnr_{1+k})\cap P_k(\supcnr_{2+k})$ is controllable with respect to $L(G_k)$ and $\Sigma_{k,u}$, then $\supcnr_{1+k} \parallel \supcnr_{2+k} = \supccnr(K, L, (\Sigma_{1,u}, \Sigma_{2,u}, \Sigma_{k,u}))$, for $L=L(G_1\| G_2\| G_k)$. \QED
  \end{theorem}

  The question remains what happens if the intersection of the projections is not controllable as required in Theorem~\ref{thm2b}. We believed in the past that a solution to the coordination control problem (in terms of a conditionally controllable sublanguage) can only be computed as a product of languages in some special cases, where sufficient conditions such as those presented in~\cite{JDEDS} hold.  However, due to the presented relaxation, it becomes clear that such a solution can always be computed in a distributed way for prefix-closed specifications. Namely, it suffices to make the resulting language $P_k(\supcnr_{1+k}) \cap P_k(\supcnr_{2+k})$ controllable with respect to $L(G_k)$, as required in Theorem~\ref{thm2b}, by computing a supervisor for it as the following result suggests. 

  \begin{proposition}\label{prop2}
    Consider the setting of Problem~\ref{problem:relaxed} and the languages defined in~(\ref{eqCNr}). Let $\supcn'_k$ denote the language $\supcn(P_k(\supcnr_{1+k}) \cap P_k(\supcnr_{2+k}), L(G_k), \Sigma_{k,u})$. If the languages $\supcnr_{i+k}$ and $\supcn_{k}'$ are synchronously nonconflicting (e.g., prefix-closed) for $i=1,2$, then $\supcn'_k \parallel \supcnr_{1+k} \parallel \supcnr_{2+k}$ is a relaxed conditionally controllable sublanguage of $K$. \QED
  \end{proposition}

  Interestingly, we rediscover this way the role of a supervisor for the coordinator that is postponed to the end of the coordination control synthesis and it is used only when needed, i.e., when controllability of $P_k(\supcnr_{1+k}) \cap P_k(\supcnr_{2+k})$ with respect to $L(G_k)$ does not hold. Moreover, the supervisor $\supcn_k'$ can be computed in a distributed way, i.e., we can compute $\supcn'_k = \supcnr(P_k(\supcnr_{1+k}), L_k, \Sigma_{k,u}) \cap \supcn(P_k(\supcnr_{2+k}), L_k, \Sigma_{k,u})$. Their composition (here intersection) is then never computed and they operate locally in conjunction with local supervisors $\supcnr_{i+k}$, $i=1,2$.
  
  Compared to Theorem~\ref{thm4}, Proposition~\ref{prop2} only states that the result is a relaxed conditionally controllable sublanguage of the specification. We do not know whether the supremality also holds here. Nevertheless, we should point out that the languages $\supcnr_{i+k}$, $i=1,2$, actually form a solution of Problem~\ref{problem1} as discuss in Section~\ref{sec:comparison} below, even though their parallel composition is not necessarily relaxed conditionally controllable. From this point of view, to decrease the language by the composition with an additional coordinator, $\supcn_k'$, is not relevant for us. This demonstrates the already mentioned drawback of the (relaxed) coordination control framework---the restriction of the set of possible solutions, compared to the set of possible solutions for Problem~\ref{problem1}; see more details in Section~\ref{sec:comparison} below.

  To compare the solutions of Problem~\ref{problem1} and those of (relaxed) coordination control problem, we need the following.
  \begin{theorem}[Inclusion in the optimal solution]\label{thm_inc_opt_sol}
    Consider the setting of Problem~\ref{problem:relaxed} and the languages defined in~(\ref{supccn}), (\ref{eqCNr}). Then the language $\supccnr \subseteq \supcnr_{1+k} \parallel \supcnr_{2+k}$, i.e., the supremal relaxed conditionally controllable sublanguage is included in our distributed solution. If $\supcnr_{1+k}$ and $\supcnr_{2+k}$ are synchronously nonconflicting (e.g., prefix-closed), then the parallel composition is controllable, i.e., $\supcnr_{1+k} \parallel \supcnr_{2+k} \subseteq \supC(K, L(G_1\| G_2),\Sigma_u)$. \QED
  \end{theorem}

  We now discuss conditions that ensures optimality. The proofs are similar to those in~\cite{JDEDS}. To this end, we use the notions of output control consistency (OCC), cf.~\cite{WZ91}, or local control consistency (LCC), cf.~\cite{SB11,SB08}. 
  \begin{theorem}[Optimality conditions]\label{thm_opt_cond}
    Consider the setting of Problem~\ref{problem:relaxed} and the languages defined in~(\ref{eqCNr}). If the projection $P_{i+k}$ is an $L(G_1\| G_2)$-observer and OCC (LCC) for $L(G_1\| G_2)$, $i=1,2$, then the parallel composition of the supervisors contains the optimal solution, i.e., $\supC(K, L(G_1\|G_2), \Sigma_{u}) \subseteq \supCr_{1+k} \parallel \supCr_{2+k}$. \QED
  \end{theorem}

  Note that we assume that $P_{i+k}$ is an $L(G_1\|G_2)$-observer and OCC (LCC) for $L(G_1\|G_2)$. However, we do not want to compute the language $L(G_1\|G_2)$. Instead, we can assume that (i) $P^{i+k}_k$ is an $(P^{i+k}_i)^{-1}(L(G_i))$-observer and (ii) $P_{i+k}$ is OCC (LCC) for the language $P_{i+k}^{-1}(L(G_i\| G_k))$. These can be ensured by an appropriate extension of the alphabet $\Sigma_k$, cf.~\cite{pcl08,pcl12} and \cite{SB11}, respectively.

\section{Comparison with the relaxed framework}\label{sec:comparison}
  Here we highlight the difference between Problem~\ref{problem1} and the relaxed coordination control problem. On one hand, we show the drawback of relaxed coordination control that prevents it from fully solving Problem~\ref{problem1}. On the other hand, we appreciate it since it can be used to find a solution. We now point out the main difference between the two problems.
  
  Let $L=L(G_1\| G_2)$ be a plant, and let $K\subseteq L$ denote a specification as in Problem~\ref{problem1}. Consider the computation of supervisors $\supcnr_{1+k}$ and $\supcnr_{2+k}$ as defined in (\ref{eqCNr}). For simplicity, we assume that the languages under consideration are prefix-closed. Assume that the parallel composition $\supcnr_{1+k} \parallel \supcnr_{2+k}$ is not relaxed conditionally controllable. This is not an unrealistic assumption. Then these two languages do not form a solution in the relaxed coordination control framework. However, note that (a generator for) $\supcnr_{i+k}$ is a supervisor for the plant $G_i\| G_k$, since $\supcnr_{i+k}$ is controllable with respect to $G_i\| G_k$ by definition. 

  Let $K' = \supcnr_{1+k} \parallel \supcnr_{2+k}$. Then $K'$ is conditionally decomposable and controllable with respect to the original plant, and
  \[
    L_m(\supcnr_{1+k}/(G_1\|G_k)) \| L_m(\supcnr_{2+k}/(G_1\|G_k)) = K'
  \]
  which means that $\supcnr_{i+k}$, $i=1,2$, is a solution of Problem~\ref{problem1} that cannot be achieved in the relaxed coordination control framework, since $K'$ is not relaxed conditionally controllable by definition. (Here we slightly abuse the notation and use $\supcnr_{i+k}$ to denote both the generator and the language of that generator, depending on the context.) 
  
  The inclusion 
  $
    L_m(\supcnr_{1+k} / (G_1\| G_k)) \subseteq P_{1+k}(K')
  $
  is required in the definition of the relaxed coordination control problem. But it does not hold in general, therefore we do not assume the inclusion to be satisfied in Problem~\ref{problem1}. This is the main difference between Problem~\ref{problem1} and the (relaxed) coordination control problem. 

  The following example shows that there exist solutions that cannot be achieved in the relaxed framework.
  \begin{example}
    Let $G_1$ and $G_2$ be plants as shown in Fig.~\ref{figexl1plants}, and let $K$ denote the specification with the generator depicted in Fig.~\ref{figexl1spec}.
    \begin{figure}
      \centering
      \includegraphics[scale=.4]{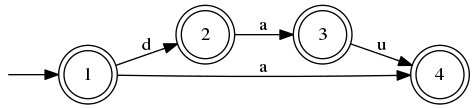}
      \qquad
      \includegraphics[scale=.4]{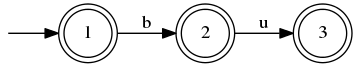}
      \caption{Plants $G_1$ and $G_2$}
      \label{figexl1plants}
    \end{figure}
    \begin{figure}
      \centering
      \includegraphics[scale=.4]{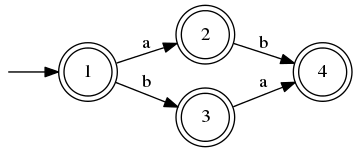}
      \caption{The specification}
      \label{figexl1spec}
    \end{figure}
    Let $E_k = \{u\}$. Notice that $K$ is conditionally decomposable. The coordinator $G_k$ is depicted in Fig.~\ref{figexl1coord}.
    \begin{figure}
      \centering
      \includegraphics[scale=.4]{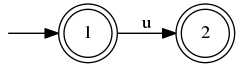}
      \caption{The coordinator $G_k$}
      \label{figexl1coord}
    \end{figure}
    Then
    $P_{1+k}(K) = \overline{\{a\}}$ and
    $P_{2+k}(K) = \overline{\{b\}}$ and it is not hard to see that the language
    $P_{2+k}(K)$ is not controllable with respect to $G_1 \| G_k = G_1$, hence
    $K$ is not relaxed conditionally controllable. 
    If we use the relaxed coordination control framework to compute supervisors, we obtain $\supC_{1+k}$ and $\supC_{2+k}$ depicted in Fig.~\ref{figexl1ccsups}.
    \begin{figure}
      \centering
      \includegraphics[scale=.4]{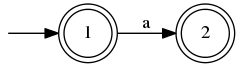}
      \qquad
      \includegraphics[scale=.4]{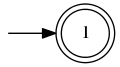}
      \caption{The supervisors $\supC_{1+k}$ and $\supC_{2+k}$ computed in the relaxed coordination control framework}
      \label{figexl1ccsups}
    \end{figure}
    This is not the optimal solution, i.e., their parallel composition is not equal to $K$.
    On the other hand, we can obtain the specification if we consider infimal superlanguages of $P_{i+k}(K)$, namely for $i=2$. Then we obtain the supervisors depicted in Fig.~\ref{figexl1infopt} that form the optimal solution.
    \begin{figure}
      \centering
      \includegraphics[scale=.4]{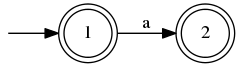}
      \qquad
      \includegraphics[scale=.4]{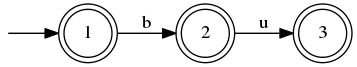}
      \caption{The optimal supervisors $S_1$ and $S_2$ for the problem obtained by considering infimal superlanguages}
      \label{figexl1infopt}
    \end{figure} \QEDopen
  \end{example}

  However, if the solution of Problem~\ref{problem1} does not exist, the relaxed coordination control framework allows us to compute a sublanguage of the specification for which a solution exists. The other advantage of this is that we can compare the obtained solution with the supremal relaxed conditionally controllable sublanguage of the specification; see below.

\section{Application of the relaxed framework}\label{sec:app}
  Consider the setting of Problem~\ref{problem1}. Theorem~\ref{th:controlsynthesissafety} says that if $K$ is relaxed conditionally controllable, then there exist nonblocking and nonconflicting supervisors $S_{1}$ and $S_{2}$ such that 
    $
      L_m(S_{1}/[G_1 \| G_k]) \parallel L_m(S_{2}/[G_2 \| G_k]) = K\,.
    $ 
  Recall the languages $\supcnr_{1+k}$ and $\supcnr_{2+k}$ defined above in (\ref{eqCNr}). By definition, we immediately have that
  \[
    \supcnr_{1+k} \parallel \supcnr_{2+k} \subseteq P_{1+k}(K) \parallel P_{2+k}(K) = K
  \]
  i.e., it is potentially a solution of Problem~\ref{problem1}. We now show that it is actually a solution and compare it with the optimal solution of the relaxed framework. The following is basically Theorem~\ref{thm_inc_opt_sol}.

  \begin{theorem}\label{optimality}
    Consider the setting of Problem~\ref{problem1} and the languages defined in~(\ref{supccn}), (\ref{eqCNr}). If the supervisors $\supcnr_{1+k}$ and $\supcnr_{2+k}$ are synchronously nonconflicting (e.g., prefix-closed), then their parallel composition is controllable. \QED
  \end{theorem}

  If the languages (supervisors) in the previous theorem are conflicting, we still have a solution for the global plant $G_1\| G_2$, namely, the language $M=trim(S_1\| S_2)$, where $trim$ denotes the nonblocking generator of $S_1\| S_2$. Then $M$ is a supervisor such that $L_m(M/G) = M$ and $L(\overline{M}/G) = \overline{M}$, so it can be considered as a solution of the monolithic case with the supervisor $M$.
  
  However, our aim is not to compute the parallel composition of supervisors $S_1$ and $S_2$ to obtain a single huge supervisor, but rather to distribute the supervision to local plants. Taking a look at the previous theorem, we can notice that $L_m(S_1/G_1\|G_k) \parallel L_m(S_2/G_2\|G_k) = L_m(S_1\| S_2)$. However, if the supervisors $S_1$ and $S_2$ are conflicting, we only have that $L(S_1/G_1\|G_k) \parallel L(S_2/G_2\|G_k) \supsetneq \overline{L_m(S_1/G_1\|G_k) \parallel L_m(S_2/G_2\|G_k)} = \overline{L_m(S_2\| S_2)}$, i.e., the overall supervised closed-loop system is blocking. To solve nonblockingness here, we can use the language
  \begin{align}\label{eq2}
    L_C=\supcn(P_0(\supcnr_{1+k})\parallel P_0(\supcnr_{2+k}), 
      \overline{P_0(\supcnr_{1+k})}\parallel \overline{P_0(\supcnr_{2+k})},\ \Sigma_{0,u})
  \end{align}
  where the projection $P_0$ is a $\supcn_{i+k}$-observer, $i=1,2$, which serves as a coordinator for nonconflictingness. The following can be proved similarly as in~\cite{JDEDS}, based on the results of~\cite{FLT}.
  
  \begin{theorem}[\cite{JDEDS}]\label{thm22}
    Consider the notation of Problem~\ref{problem1} and the languages defined in~(\ref{eqCNr}) and~(\ref{eq2}). Then the language
    $
      \overline{\supcn_{1+k}  \parallel \supcn_{2+k}\parallel L_C} = 
      \overline{\supcn_{1+k}} \parallel \overline{\supcn_{2+k}} \parallel \overline{L_C}
    $
    is controllable with respect to the plant $G_1 \| G_2$. \QED
  \end{theorem}
  
  We can now summarize the method as an algorithm.
  \begin{algorithm}[Solving Problem~\ref{problem1}]\label{alg}
    Consider the above.
    \begin{enumerate}
      \item Check whether Problem~\ref{problem1} has a solution using Proposition~\ref{lem1}. If so, stop; otherwise, continue.
      \item Compute $\supcnr_{1+k}$ and $\supcnr_{2+k}$ as defined in (\ref{eqCNr}).
      \item Let $\Sigma_0:=\Sigma_k$ and $P_0:=P_k$.
      \item Extend the alphabet $\Sigma_0$ so that the projection $P_0$ is both a $\supcnr_{1+k}$- and a $\supcnr_{2+k}$-observer.
      \item Define the coordinator $C$ as the minimal nonblocking generator such that $L_m(C) = L_C$ from~(\ref{eq2}).
    \end{enumerate}
  \end{algorithm}
  See more comments on this algorithm in~\cite{JDEDS}.

\section{MRI Scanner}\label{sec:MRI}
  To demonstrate our approach on an industrial example, we consider the model and specification of an MRI scanner presented by Theunissen~\cite{theunissen}. The plant consists of four parts
  \[
    \text{VAxis}  \parallel  \text{HAxis} \parallel  \text{HVNormal}  \parallel \text{UI}
  \]
  where each part is again a composition of several smaller parts. However, we do not go into these details and consider these parts as the four subsystems that form the plant. The specification consists of the parts
  \[
    \text{VReq}  \parallel  \text{HReq}  \parallel  \text{HVReq}  \parallel \text{UIReq}
  \]
  which do not exactly correspond the the four parts of the plant. For all computations, we used the C++ library \texttt{libFAUDES}~\cite{libfaudes}. We now observe the following.

  (1) VReq is a specification that concerns only the plant VAxis. This is a simple case with one plant, hence a monolithic approach was used. The specification has 12 states and 44 transitions, the plant has 15 states and 50 transitions, and the computed supervisor has 15 states and 36 transitions.

  (2) Similarly, HReq is a specification that concerns only the plant HAxis. Again, the monolithic supervisor was computed. The specification has 112 states and 736 transitions, the plant has 128 states and 1002 transitions, and the computed supervisor consists of 80 states and 320 transitions.

  (3) The specification HVReq concerns the plant $\text{VAxis} \parallel \text{HAxis} \parallel \text{HVNormal}$. The specification consists of 7 states and 35 transitions, the plant HVNormal consists of 1 state and 1 transition (the other two parts of the plant are mentioned above). We computed a coordinator consisting of 160 states and 1287 transitions and three supervisors with 516, 1132 and 283 states and 3395, 10298 and 1692 transitions, respectively. It was verified that their parallel composition is nonblocking and corresponds to the supremal controllable sublanguage of the specification HVReq with respect to the plant $\text{VAxis}\parallel \text{HAxis} \parallel \text{HVNormal}$.
    
  (4) The specification UIReq concerns the whole plant $\text{VAxis} \parallel \text{HAxis} \parallel \text{HVNormal} \parallel \text{UI}$. Therefore it is the most interesting part to demonstrate our approach. Let us mention that UIReq consists of 256 states and 2336 transitions, and UI consists of 2 states and 15 transitions (the rest is described above). We computed a coordinator with 4 states and 30 transitions, and four supervisors with 432, 768, 12, and 96 states and 3488, 6652, 74 and 808 transitions, respectively. 

  The overall minimal monolithic supervisor for the whole system would consist of 68672 states and 616000 transitions. We have verified that our solution is an optimal solution to Problem~\ref{problem1}. We have verified that it forms a solution in both coordination control architectures discussed in the paper. Due to the modeling skills of the modeler and the nature of the model, there were no problems with blockingness at any step of the computation. All the computed supervisors are nonblocking and nonconflicting, thus no coordinator for nonblockingness was needed.

\section{Conclusion and further discussion} \label{sec:conclusion}
  The results are formulated for general, non-prefix-closed languages, for which synchronous nonconflictingness is required. One could find it an issue, but these assumptions are trivially satisfied for, e.g., prefix-closed languages. To verify whether a synchronous product (of an unspecified number) of generators is synchronously nonconflicting is PSPACE-complete~\cite{rohloff}. It is only the worst case. Some optimization techniques exist, e.g.~\cite{FlordalM2006}, or a maximal nonconflicting sublanguage can be computed~\cite{Chen1991105}. The good news of PSPACE-completeness is that it is computable in polynomial space.

  All concepts and results can be extended from the generic case $n=2$ in a straightforward manner. However, with an increasing number of components it is likely that the coordinator will, in some situations, grow, e.g., many events will have to be included into the coordinator alphabet to make the global specification conditionally decomposable. Therefore we have recently proposed a multilevel coordination control architecture~\cite{cdc2013}. The approach of this paper can easily be implemented in such a multilevel structure.

  A challenging and important problem is to find a convenient alphabet $\Sigma_k$, from which the coordinator is computed.  A natural step is to take a minimal such alphabet. There are two issues with this choice. First, to compute the minimal $\Sigma_k$ is NP-hard, second, there exist examples showing evidence that the use of a minimal alphabet results in no solutions (in empty supervisors). This means that we need to find a larger alphabet for which a solution can be obtained. This is always possible, since $\Sigma_k$ can be taken as the global alphabet. However, the aim is to characterize a reasonable choice of the alphabet that would provide nonempty solutions.
  
  In the future, we plan to take into account communication delays and losses and to consider more complicated forms of communications among local controllers.

\section*{Acknowledgment}
  The research was supported by RVO 67985840, by M\v{S}MT in project MUSIC (grant LH13012), by GA{\v C}R in project GA15-02532S and by the DFG in project DIAMOND (Emmy Noether grant KR~4381/1-1).

\clearpage
\appendix
Here we provide the proofs of Lemma~\ref{thm:char}, Proposition~\ref{lem1} and Theorem~\ref{thmMain}.
The results concerning the relaxed coordination control framework can be found in the technical report~\cite{relaxed}.

  \begin{proofof}{Lemma~\ref{thm:char}}
    Consider the setting of Problem~\ref{problem1}. There exist nonblocking and nonconflicting supervisors $S_{1}$ and $S_{2}$ such that the distributed closed-loop system satisfies 
    $
      L_m(S_{1}/[G_1 \| G_k]) \| L_m(S_{2}/[G_2 \| G_k]) = K
    $ 
    if and only if
      $L_m(S_1/(G_1\| G_k)) \| P_k(S_2) = P_{1+k}(K)$,
      $L_m(S_2/(G_2\| G_k)) \| P_k(S_1) = P_{2+k}(K)$, and
      $S_1$ and $S_2$ are nonblocking and nonconflicting supervisors with respect to $G_1\| G_k$ and $G_2\| G_k$, respectively.
  \end{proofof}
  
  \begin{proof}
    (If) From the assumptions we obtain that 
    \begin{align*}
      K & = P_{1+k}(K) \| P_{2+k}(K)  \\
        & = L_m(S_1/(G_1\| G_k)) \| P_k(S_2) \| L_m(S_2/(G_2\| G_k)) \| P_k(S_1)\\
        & = L_m(S_1/(G_1\| G_k)) \| P_k(S_1) \| L_m(S_2/(G_2\| G_k)) \| P_k(S_2)\\
        & = L_m(S_1/(G_1\| G_k)) \| L_m(S_2/(G_2\| G_k))
    \end{align*}
    and similarly for
    \begin{align*}
      \overline{K} & = \overline{P_{1+k}(K)} \| \overline{P_{2+k}(K)}\\
        & = \overline{L_m(S_1/(G_1\| G_k)) \| P_k(S_2)} \| \overline{L_m(S_2/(G_2\| G_k)) \| P_k(S_1)}\\
        & = L(S_1/(G_1\| G_k)) \| \overline{P_k(S_1)} \| L(S_2/(G_2\| G_k)) \| \overline{P_k(S_2)}\\
        & = L(S_1/(G_1\| G_k)) \| L(S_2/(G_2\| G_k)).
    \end{align*}
    
    (Only if) By the application of projection $P_{1+k}$ to $L_m(S_{1}/[G_1 \| G_k]) \| L_m(S_{2}/[G_2 \| G_k]) = K$, we obtain that $P_{1+k}(K) = L_m(S_1/(G_1\| G_k)) \| P_k(S_2)$, cf. Lemma~\ref{lemma:Wonham} below. Similarly for $P_{2+k}(K)$.
  \end{proof}
  
  \medskip
  \begin{lemma}[\cite{Won12}]\label{lemma:Wonham}
    Let $P_k : \Sigma^*\to \Sigma_k^*$ be a projection, and let $L_i \subseteq \Sigma_i^*$, where $\Sigma_i\subseteq \Sigma$, for $i=1,2$, and $\Sigma_1\cap \Sigma_2 \subseteq \Sigma_k$. Then $P_k(L_1\| L_2)=P_k(L_1) \| P_k(L_2)$. 
  \end{lemma}

  \bigskip

  \begin{proofof}{Proposition~\ref{lem1}}
    Consider the setting of Problem~\ref{problem1}. If the specification $K$ is prefix-closed, then there exists a solution of Problem~\ref{problem1} if and only if the languages 
      $T_1 = \infCO(P_{1+k}(K), G_1\| G_k)$ and
      $T_2 = \infCO(P_{2+k}(K), G_2\| G_k)$
    satisfy equations~(\ref{sol1}).
  \end{proofof}
  
  \begin{proof}
    (If) Assume that $T_1\, \cap\, (P^{2+k}_k)^{-1}P_k(T_2) = P_{1+k}(K)$ and $T_2\, \cap\, (P^{2+k}_k)^{-1}P_k(T_1) = P_{2+k}(K)$. Since $T_1$ and $T_2$ are prefix-closed, hence nonblocking and nonconflicting, and controllable with respect to the corresponding plants, that is, they are supervisors for those plants, the implication follows from Lemma~\ref{thm:char}.
    
    (Only if) Assume that there is a solution of Problem~\ref{problem1}. By Lemma~\ref{thm:char}, there exist supervisors $S_1$ and $S_2$ satisfying equations (\ref{sol1}). Since $S_i$ is a supervisor, its language (for simplicity also denoted by $S_i$) is controllable with respect to $G_i\| G_k$. Moreover, $S_i$ contains $P_{i+k}(K)$, hence $T_i\subseteq S_i$, for $i=1,2$. Then,
    \begin{align*}
      P_{1+k}(K) & = P_{1+k}(K) \| P_k(P_{2+k}(K)) \subseteq T_1 \| P_k(T_2) \subseteq S_1\| P_k(S_2) = P_{1+k}(K).
    \end{align*}
    Thus, $T_1$ and $T_2$ satisfy equations (\ref{sol1}).
  \end{proof}

  \bigskip
  
  \begin{proofof}{Theorem~\ref{thmMain}}
    Consider the setting of Problem~\ref{problem1}. There exists a solution for specification $K$ if and only if there exists a solution for its prefix-closure $\overline{K}$.
  \end{proofof}
  
  To prove this theorem, we first show the simple direction.
  \begin{lemma}
    Consider the setting of Problem~\ref{problem1}. If there exists a solution for specification $K$, then there exists a solution for its prefix-closure $\overline{K}$.
  \end{lemma}
  \begin{proof}
    Let $S_1$ and $S_2$ denote the solution for specification $K$. If $L_m(S_1/(G_1\| G_k)) \| P_k(S_2) = P_{1+k}(K)$, then 
    \begin{align*}
      \overline{P_{1+k}(K)} 
      & = \overline{L_m(S_1/(G_1\| G_k)) \| P_k(S_2)} \\
      & = \overline{L_m(S_1/(G_1\| G_k))} \| \overline{P_k(S_2)} \\
      & = L(S_1/(G_1\| G_k)) \| \overline{P_k(S_2)}
    \end{align*}
    by nonconflictingness and nonblockingness of the supervisors. Similarly for the other equation. Thus, $\overline{S_1}$ and $\overline{S_2}$ form a solution for $\overline{K}$.
  \end{proof}

  Thus, it remains to show that the existence of a solution for the prefix-closure of $K$ implies the existence of a solution for $K$ itself.
  \begin{lemma}
    Consider the setting of Problem~\ref{problem1}. If there exists a solution for the prefix-closure $\overline{K}$ of specification $K$, then there exists a solution for specification $K$.
  \end{lemma}
  \begin{proof}
    Let $G_{S_1}$ and $G_{S_2}$ denote the generators of the solution for the prefix-closed language $\overline{K}$, that is,
    \begin{align*}
      L(G_{S_1}/(G_1\| G_k)) &\| P_k(L(G_{S_2})) = \overline{P_{1+k}(K)}\\
      &\text{ and } \\
      L(G_{S_2}/(G_2\| G_k)) &\| P_k(L(G_{S_1})) = \overline{P_{2+k}(K)}\,.
    \end{align*}
    For $i=1,2$, we now define the marking on $G_{S_1}$ and $G_{S_2}$ as follows:
    \[
      L_m(G_{S_i}) = P_{i+k}(K) \cup \left( L(G_{S_i}) \setminus \overline{P_{i+k}(K)} \right)\,.
    \]
    Note that $L_m(G_{S_i})\subseteq L(G_{S_i})$, since $P_{i+k}(K)\subseteq L(G_{S_i})$. 
    
    We show that $G_{S_i}$ is nonblocking, that is, $L(G_{S_i})\subseteq \overline{L_m(G_{S_i})}$. To this end, let $w\in L(G_{S_i})$. If $w\notin \overline{P_{i+k}(K)}$, then $w\in L_m(G_{S_i})$ by definition. If $w\in \overline{P_{i+k}(K)}$, then there exists a word $v$ such that $wv\in P_{i+k}(K)\subseteq L_m(G_{S_i})$. Hence $w\in \overline{L_m(G_{S_i})}$.
    Thus, $\overline{L_m(G_{S_i})}=L(G_{S_i})$. 
    
    Moreover, $L_m(G_{S_i})$ is controllable with respect to the plant $G_i \| G_k$, because, by assumption, $L(G_{S_i})$ is. Thus, $G_{S_i}$ is a marking nonblocking supervisor for $G_i\| G_k$.
    Furthermore, we have that
    \begin{align*}
      P_{1+k}(K) 
       & \subseteq L_m(G_{S_1}/(G_1\| G_k)) \| P_k(L_m(G_{S_2}))\\
       & \subseteq L(G_{S_1}/(G_1\| G_k)) \| P_k(L(G_{S_2}))\\
       & = \overline{P_{1+k}(K)}
    \end{align*}
    
    To show that \[L_m(G_{S_1}/(G_1\| G_k)) \| P_k(L_m(G_{S_2}))\subseteq P_{1+k}(K)\] note that if there exists $w\in \overline{P_{1+k}(K)} \setminus P_{1+k}(K)$, then $w$ does not belong to $L_m(G_{S_1})$ by definition. Thus, we have shown that $L_m(G_{S_1}/(G_1\| G_k)) \| P_k(L_m(G_{S_2})) = P_{1+k}(K)$. The case for $G_{S_2}$ is analogous.

    Finally, 
    \begin{align*}
      \overline{K}
      & = \overline{P_{1+k}(K) \| P_{2+k}(K)} \\
      & = \overline{L_m(G_{S_1}/(G_1\| G_k)) \| P_k(L_m(G_{S_2})) 
          \| L_m(G_{S_2}/(G_2\| G_k)) \| P_k(L_m(G_{S_1}))}\\
      & = \overline{L_m(G_{S_1}/(G_1\| G_k)) \| L_m(G_{S_2}/(G_2\| G_k))}\\
      & \subseteq \overline{L_m(G_{S_1}/(G_1\| G_k))} \| \overline{L_m(G_{S_2}/(G_2\| G_k))}\\
      & = L(G_{S_1}/(G_1\| G_k)) \| L(G_{S_2}/(G_2\| G_k))\\
      & = \overline{P_{1+k}(K)} \| \overline{P_{2+k}(K)} = \overline{K}
    \end{align*}
    which shows that the supervisors are also nonconflicting.
  \end{proof}


\begin{thebibliography}{10}

\bibitem{libfaudes}
\texttt{libFAUDES} -- a software library for supervisory control.
\newblock http://www.rt.eei.uni-erlangen.de/FGdes/faudes/index.html.

\bibitem{schuppen:etal:2011:ejc}
O.~Boutin, J.~Komenda, T.~Masopust, N.~Pambakian, J.H. van Schuppen, P.L.
  Kempker, and A.C.M. Ran.
\newblock Control of distributed systems: Tutorial and overview.
\newblock {\em Eur. J. Control}, 17(5-6):579--602, 2011.

\bibitem{pcl12}
H.J. Bravo, A.E.C. {Da Cunha}, P.N. Pena, R.~Malik, and J.E.R. Cury.
\newblock Generalised verification of the observer property in discrete event
  systems.
\newblock In {\em Proc. of WODES}, pages 337--342, Mexico, 2012.

\bibitem{caiCDC13}
K.~Cai, R.~Zhang, and W.M. Wonham.
\newblock On relative observability of discrete-event systems.
\newblock In {\em Proc. of 52nd {IEEE} Conference on Decision and Control
  ({CDC})}, pages 7285--7290, Italy, 2013.

\bibitem{CL08}
C.G. Cassandras and S.~Lafortune.
\newblock {\em Introduction to discrete event systems}.
\newblock Springer, second edition, 2008.

\bibitem{Chen1991105}
E.~Chen and S.~Lafortune.
\newblock On nonconflicting languages that arise in supervisory control of
  discrete event systems.
\newblock {\em Systems Control Lett.}, 17(2):105--113, 1991.

\bibitem{LauriePhilippe}
P.~Darondeau and L.~Ricker.
\newblock Distributed control of discrete-event systems: A first step.
\newblock {\em ToPNoC}, 6:24--45, 2012.

\bibitem{FLT}
L.~Feng.
\newblock {\em Computationally Efficient Supervisor Design for Discrete-Event
  Systems}.
\newblock PhD thesis, University of Toronto, 2007.

\bibitem{FW2006}
L.~Feng and W.M. Wonham.
\newblock Computationally efficient supervisor design: Abstraction and
  modularity.
\newblock In {\em Proc. of WODES}, pages 3--8, USA, 2006.

\bibitem{FengWonham}
L.~Feng and W.M. Wonham.
\newblock On the computation of natural observers in discrete-event systems.
\newblock {\em Discrete Event Dyn. Syst.}, 20:63--102, 2010.

\bibitem{FlordalM2006}
H.~Flordal and R.~Malik.
\newblock Modular nonblocking verification using conflict equivalence.
\newblock In {\em Proc. of WODES}, pages 100--106, 2006.

\bibitem{FM2009}
H.~Flordal and R.~Malik.
\newblock Compositional verification in supervisory control.
\newblock {\em SIAM J. Control Optim.}, 48:1914--1938, 2009.

\bibitem{HT2006}
R.~C. Hill and D.~M. Tilbury.
\newblock Modular supervisory control of discrete event systems with
  abstraction and incremental hierarchical construction.
\newblock In {\em Proc. of WODES}, pages 399--406, USA, 2006.

\bibitem{hubbard:caines:2002}
P.~Hubbard and P.E. Caines.
\newblock Dynamical consistency in hierarchical supervisory control.
\newblock {\em {IEEE} Trans. Automat. Contr.}, 47(1):37--52, 2002.

\bibitem{scl12}
J.~Komenda, T.~Masopust, and J.H. van Schuppen.
\newblock On conditional decomposability.
\newblock {\em Systems Control Lett.}, 61(12):1260--1268, 2012.

\bibitem{cdc2013}
J.~Komenda, T.~Masopust, and J.H. van Schuppen.
\newblock Multilevel coordination control of modular {DES}.
\newblock In {\em Proc. of CDC}, pages 6323--6328, 2013.

\bibitem{JDEDS}
J.~Komenda, T.~Masopust, and J.H. van Schuppen.
\newblock Coordination control of discrete-event systems revisited.
\newblock {\em Discrete Event Dyn. Syst.}, 25:65--94, 2014.

\bibitem{case15}
J.~Komenda, T.~Masopust, and J.H. van Schuppen.
\newblock Combined top-down and bottom-up approach to multilevel supervisory
  control.
\newblock Available at http://arxiv.org/abs/1502.07328, 2015.

\bibitem{KomendaMS14a}
J.~Komenda, T.~Masopust, and J.H. van Schuppen.
\newblock Relative observability in coordination control.
\newblock In {\em Proc. of CASE}, 2015.
\newblock Accepted.

\bibitem{relaxed}
J.~Komenda, T.~Masopust, and J.H. van Schuppen.
\newblock A relaxed framework for coordination control of discrete-event
  systems.
\newblock Available at http://arxiv.org/abs/1501.07859, 2015.

\bibitem{LinSSWS14}
L.~Lin, A.~Stefanescu, R.~Su, W.~Wang, and A.R. Shehabinia.
\newblock Towards decentralized synthesis: Decomposable sublanguage and joint
  observability problems.
\newblock In {\em Proc. of ACC}, pages 2047--2052, USA, 2014.

\bibitem{pcl08}
P.N. Pena, J.E.R. Cury, and S.~Lafortune.
\newblock Polynomial-time veriﬁcation of the observer property in
  abstractions.
\newblock In {\em Proc. of ACC}, pages 465--470, USA, 2008.

\bibitem{pena2010}
P.N. Pena, J.E.R. Cury, R.~Malik, and S.~Lafortune.
\newblock Efficient computation of observer projections using {OP}-verifiers.
\newblock In {\em Proc. of WODES}, pages 416--421, Germany, 2010.

\bibitem{RW87}
P.J. Ramadge and W.M. Wonham.
\newblock Supervisory control of a class of discrete event processes.
\newblock {\em SIAM J. Control Optim.}, 25(1):206--230, 1987.

\bibitem{rohloff}
K.~Rohloff and S.~Lafortune.
\newblock {PSPACE}-completeness of modular supervisory control problems.
\newblock {\em Discrete Event Dyn. Syst.}, 15:145--167, 2005.

\bibitem{SB08}
K.~Schmidt and C.~Breindl.
\newblock On maximal permissiveness of hierarchical and modular supervisory
  control approaches for discrete event systems.
\newblock In {\em Proc. of WODES}, pages 462--467, Sweden, 2008.

\bibitem{SB11}
K.~Schmidt and C.~Breindl.
\newblock Maximally permissive hierarchical control of decentralized discrete
  event systems.
\newblock {\em {IEEE} Trans. Automat. Contr.}, 56(4):723--737, 2011.

\bibitem{KS}
K.~Schmidt, T.~Moor, and S.~Perk.
\newblock Nonblocking hierarchical control of decentralized discrete event
  systems.
\newblock {\em {IEEE} Trans. Automat. Contr.}, 53(10):2252--2265, 2008.

\bibitem{theunissen}
R.J.M. Theunissen.
\newblock {\em Supervisory Control in Health Care Systems}.
\newblock PhD thesis, Technische Universiteit Eindhoven, 2015.

\bibitem{wong98}
K.~Wong.
\newblock On the complexity of projections of discrete-event systems.
\newblock In {\em Proc. of WODES}, pages 201--206, Italy, 1998.

\bibitem{Won12}
W.M. Wonham.
\newblock Supervisory control of discrete-event systems. {U}ni. {T}oronto,
  2012.
\newblock Available at http://www.control.utoronto.ca/DES/.

\bibitem{WZ91}
H.~Zhong and W.M. Wonham.
\newblock On the consistency of hierarchical supervision in discrete-event
  systems.
\newblock {\em {IEEE} Trans. Automat. Contr.}, 35(10):1125--1134, 1990.

\end{thebibliography}
\end{document}